\documentclass[12pt,reqno]{amsart}

\advance \topmargin by -\headheight
\advance \topmargin by -\headsep
\evensidemargin \oddsidemargin
\usepackage{amsmath}
\usepackage{bbm}
\usepackage{amsfonts}
\usepackage{amssymb}
\usepackage{amsthm}
\usepackage{csquotes}
\usepackage{color}
\usepackage{mathrsfs}
\usepackage{dsfont}
\usepackage{bm}
\usepackage{cite}
\usepackage{enumerate}
\usepackage{booktabs}
\usepackage[mathscr]{eucal}
\usepackage{mathtools}
\usepackage{url}

\numberwithin{equation}{section}
\newtheorem{theorem}{Theorem}[section]
\newtheorem{proposition}[theorem]{Proposition}
\newtheorem{lemma}[theorem]{Lemma}
\newtheorem{corollary}[theorem]{Corollary}
\theoremstyle{definition}

\theoremstyle{remark}

\newcommand{\T}{\mathbb T}
\newcommand{\Z}{\mathbb Z}
\newcommand{\R}{\mathbb R}
\newcommand{\N}{\mathbb N}
\newcommand{\ML}{\mathrm{ML}}
\newcommand{\dd}{\,\mathrm{d}}
\title{Near optimal gap in the Lonely Runner Conjecture}
\author{Xiyu Hu}
\address{School of Mathematical Sciences, University of Chinese Academy of Sciences, China}
\email{hxypqr@gmail.com}
\author{Yifan Jing}
\address{Department of Mathematics, The Ohio State University, Columbus OH, 43210, USA}
\email{jing.245@osu.edu}

\begin{document}
\begin{abstract}
For a set $V$ of $n$ distinct positive integers, let $\ML(V)=\max_{t\in\R/\Z}\min_{v\in V}\|vt\|$, where $\|x\|$ is the distance from $x$ to the nearest integer. We prove that there is an absolute constant $c>0$ such that for every $n\ge2$ and every $V$,
\[
\ML(V)>\frac{1+c}{2n}.
\]
The proof combines a refinement of the large GCD argument of Green and Walker with a new physical spillover inequality for interval multiplicities.
\end{abstract}
\maketitle

\tableofcontents

\section{Introduction}

The lonely runner conjecture has its origins in the study of Diophantine approximation, dating back to the 1965 doctoral thesis of Wills~\cite{Wills}. As Wills later recalled, his research was motivated by Hermann Weyl's equidistribution theorem of 1916~\cite{Weyl}, this led him to investigate the opposite extremal question of worst Diophantine approximation. Cusick~\cite{Cusick} subsequently gave an equivalent formulation in terms of view-obstruction problems after discussions with Wills at the ICM congress in Nice(Nizza) in 1970. Its name \emph{Lonely Runner Conjecture} comes from a reformulation by Goddyn~\cite{BieniaEtAl} as follows: $n+1$ runners move around a circle of unit length at distinct constant speeds, starting at the same point. The conjecture asserts that each runner is, at some time, at distance at least $1/(n+1)$ from every other runner. The problem has received many attentions in the literature, with connections to many other directions including Diophantine approximation, view-obstruction, billiard trajectory problems, chromatic numbers of distance graphs, and flows in matroids. For background of these directions and references we refer the reader to a recent survey of Perarnau and Serra \cite{PSsurvey}. 
The conjecture has also led to several related questions and variations. These include the possible values of maximum loneliness~\cite{Kravitz}, the lonely runner spectra~\cite{{GiriKravitz,JainKravitz}} and geometric formulations using lattice zonotopes~\cite{{BeckSchymura,MalikiosisSantosSchymura}}. More recently, Alon, Buci\'c and Davies~\cite{AlonBucicDavies} proved a higher-dimensional simultaneous avoidance theorem for affine linear forms, which settles Schoenberg's view-obstruction conjecture, which can be seen as a high dimensional analogy of lonely runner problem, but it does not directly connect to the original conjecture. 

Among different lines of research around the conjecture, a central question is to improve the lower bound uniformly over all speed sets. It is the question that we address in the present paper.

We work with the usual integer formulation. We use $\T=\R/\Z$ for one-dimensional torus, and let $\mu$ be the normalized Haar measure on $\T$, and let $\|x\|=\mathrm{dist}(x,\Z)$ for $x\in\T$. For a nonempty finite set $V=\{v_1,\ldots,v_n\}\subset\N$, following the notation of Kravitz~\cite{Kravitz} and Bedert~\cite{Bedert}, we define its \emph{maximum loneliness} by
\[
 \ML(V)=\max_{t\in\T}\min_{v\in V}\|vt\|.
\]
Here $n=|V|$ counts the nonzero relative speeds, so the conjectured bound is
\[
\ML(V)\ge\frac1{n+1}.
\]
This formulation is equivalent to the original conjecture for real speeds; see~\cite[Section~4]{BohmanHolzmanKleitman}. The bound would be sharp, since $\ML(\{1,\ldots,n\})=1/(n+1)$.

For general $n$, a simple union bound gives $\ML(V)\ge1/2n$. Indeed, for $v\in V$ and $0<\delta<1/2$, consider the Bohr set
\[
 B(v;\delta)=\{t\in\T:\|vt\|\le\delta\}.
\]
It has measure $2\delta$. Thus, if $2n\delta<1$, the sets $B(v;\delta)$ cannot cover $\T$, and any point outside their union would have distance greater than $\delta$ from every moving runner.
By letting $\delta$ tend to $1/2n$, we  prove the bound. 

For many decades, the known improvements over $1/(2n)$ were only of order $n^{-2}$. These include the work of Chen~\cite{Chen} and Chen and Cusick~\cite{ChenCusick}, the latter under a primality assumption. Perarnau and Serra~\cite{PerarnauSerra} subsequently proved
\[
 \ML(V)\ge\frac1{2n-2+o(1)}.
\]
Thus these estimates changed the denominator by a bounded amount, while leaving its leading term unchanged.

Tao~\cite{Tao} was the first to obtain an improvement of larger order. He showed that, for all sufficiently large $n$,
\[
 \ML(V)\ge\frac1{2n}
       +c\frac{\log n}{n^2(\log\log n)^2}
\]
for an absolute constant $c>0$. More recently, Bedert~\cite{Bedert} obtained the power-saving improvement
\[
 \ML(V)\ge\frac1{2n}+n^{-5/3-o(1)}.
\]
His proof separates speed sets according to their additive dimension. In large dimension he uses Riesz products, while in small dimension he constructs a model with controlled maximum speed. In subsequent work, Bedert and the first author~\cite{BH26} further improve this bound. Bedert~\cite{B3/2} has also communicated the improvement
\[
 \ML(V)\ge\frac1{2n}+n^{-3/2-o(1)}.
\]
All these bounds still give a relative improvement tending to zero as $n$ tends to infinity.

The purpose of this paper is to prove a constant relative improvement over the union bound. 

\begin{theorem}\label{thm:main}
There is an absolute constant $c>0$ such that, for every set $V$ of $n\ge2$ distinct positive integers,
\[
  \ML(V)>\frac{1+c}{2n}.
\]
\end{theorem}

We in fact proved a stronger statement about the measure of the times at which all runners are separated from the stationary runner. Let
\[
 g(V)=\mu\{t\in\T:\|vt\|>1/2n\text{ for every }v\in V\},
\]
and we showed that $g(V)>\kappa_0$ for some absolute constant $\kappa_0$. One may then take $c=\kappa_0/2$ in Theorem~\ref{thm:main}.

The restriction $n\ge2$ is necessary, since $\ML(\{v\})=1/2$. All constants in the proof are explicit and we do not attempt to optimize them. We give a choice of $\kappa_0$ in Section~\ref{sec:main proof}.

Our main theorem also yields a nontrivial lower bound for the \emph{pyjama problem}, solved by Manners~\cite{MannersPyjama}, which asks for the minimum number $P(\varepsilon)$ of rotated copies of the periodi stripe $(E(\varepsilon)={z\in \mathbb C:|\mathfrak{R}z|\le \varepsilon})$ needed to cover the plane. Kravitz and Leng~\cite{KravitzLeng} observed that a variant allowing both rotations and dilations is equivalent to the lonely runner problem. Consequently, Theorem~\ref{thm:main} implies that $P(\varepsilon)\ge (1+c)/(2\varepsilon)$, giving a constant relative improvement over the trivial density lower bound $P(\varepsilon)\ge 1/(2\varepsilon)$.

\subsection*{Organization}
In the next section, we will give an overview of our proof. 
Section~\ref{sec:Type I} records the overlap identities and the pair estimates for both Type I and Type II overlaps. Section~\ref{sec: spillover} proves the dilation spillover inequality. Section~\ref{sec:rational-boxes} proves a dense family with bounded rational coordinates has substantial overlap at common centers. In Section~\ref{sec:large GCD}, we recall the large GCD argument with the quantitative form needed here. Section~\ref{sec:linear ratio sparsity} refines this argument to remain a positive proportion of the speeds and proves the arithmetic input. We combine the two parts in Section~\ref{sec:main proof}.

\subsection*{AI statement}
The main strategy and mathematical framework of the paper were developed by the authors. ChatGPT 6 Astra and ChatGPT 6 Pro were used primarily to verify arguments, turn preliminary or informal proofs into rigorous proofs, and explore possible refinements. From the authors' perspective, this is an interesting example of the importance of human mathematician input in mathematical research: the models were unable to improve the lonely runner bound on their own.

The starting point came from the first author, who initially work on improving Bedert's bound by replacing dissociation dimension with a finer measure of the arithmetic complexity of the speed set. Motivated by a heuristic analogy between dissociation in Riesz products and $p$-adic frequency separation, and by the projection and multiscale methods of~\cite{DFGGL}, he explored modular smoothing and scale summation for a fixed prime $p$, obtaining partial improvements when the speed set satisfies suitable modularity conditions. To combine information across different primes and impose further constraints on speed sets not covered by these improvements, ChatGPT 6 Pro suggested introducing suitable GCD sum estimates. These partial results formed the starting point of the  collaboration, which aimed to establish a uniform improvement for arbitrary speed sets.

The authors then began working together to make this improvement uniform over all sets $V$. After many rounds of discussion with ChatGPT 6 Pro and ChatGPT 6 Astra, the method was gradually refined. In particular, the AI tools pointed the authors toward the Green--Walker refinement of the GCD sum estimates. However, these discussions did not produce any uniform results.

Motivated by recent joint work with Meng~\cite{JingMeng} and earlier work on spillover methods, the second author then refined the spillover approach developed there and established the dilation spillover inequality, Proposition~\ref{spill:pair spillover}. Some small gaps in the first version of his proof were later fixed by ChatGPT 6 Pro. To handle the remaining Type I overlaps, as defined in the overview, a slightly strengthened version of the Green--Walker theorem was needed. Once the required statement had been identified, ChatGPT 6 Pro assisted the second author in refining the Green--Walker argument to obtain the necessary form. 

The final organization, mathematical presentation, and proofs were written by the authors. ChatGPT 6 Astra was additionally used for language polishing. The figures are drawn by the authors using Ipe.

\subsection*{Acknowledgments}
The authors thank Benjamin Bedert for many useful discussions on the problem. 
The first author thanks Yangkendi Deng for detailed explanations of the techniques and underlying ideas in their work~\cite{DFGGL} on Bourgain's $L^6$ problem on the one-dimensional torus.
The second author thanks Zachary Chase for earlier discussions on Tao's Bohr set approach, Chieu-Minh Tran for discussions on applications and generalizations of the spillover argument for many years.
The second author is supported by NSF grant DMS-2503063.

\section{Proof setup and overview}\label{sec: setup}

Throughout the paper, we use standard terminology from analytic number theory, where $X\ll Y$ means $X\le CY$ for an absolute constant $C>0$, and subscripts indicate possible dependence. All logarithms are natural. The speed set $V=\{v_1,\ldots,v_n\}$ is always a set of $n$ positive integers, and integrals on $\T$ are with respect to normalized Haar measure. To study the level set, we use $x_+$ for $\max\{x,0\}$. We always assume $0<\delta<\frac{1}{2}$.

In this section,
we give an overview of the proof. Along the way, we will also give some key definitions and notation used in the later proofs.

Our proof is different from the earlier approaches by Bedert and by Tao, and we will explain it briefly in this section. From now on, fix a set $V$ of $n\ge2$ distinct positive integers. Let us first explain why an improvement over the union bound requires information about overlaps. We write
\[
 N_\delta(t)=\sum_{v\in V}\mathbbm1_{B(v;\delta)}(t),
 \qquad\text{and}\qquad
 w_V(\delta)=\int_\T(N_\delta-1)_+\dd\mu. 
\]
Thus $N_\delta$ counts the Bohr sets containing $t$, and $w_V(\delta)$ measures their overlap, counting each of the covering beyond the first. The authors think it is helpful to view them geometrically: Write $V=\{v_1,\ldots,v_n\} $, one can embed Bohr set overlaps to $X_\delta\subseteq \Z\times \T$ where $(m,t)\in X_\delta$ means $t$ is an interior point of $B(v_m;\delta)$. Under the natural projection $\pi: \Z\times\T \to \T$, $N_\delta (t) = |\pi^{-1}(t)\cap X_\delta|$ the fiber cardinality function at $t$, and the support set of $(N_\delta-k)_+$ is the $k$-level set. It is important to keep this picture in mind as we will later use layer cake representations to study overlaps. 

We observe that at the radius $\delta_0=1/2n$, the mean of $N_{\delta_0}$ is $1$. As a Consequence,
\[
    w_V(\delta_0)=\mu\{t\in\T:\|vt\|>1/2n\text{ for every }v\in V\}=g(V).
\]
In other words, at this radius the total overlap is exactly the uncovered measure. We need to show that the overlaps cannot be arbitrarily small, uniformly in the number and sizes of the speeds.
We remark that this covering viewpoint is also used by Perarnau and Serra~\cite{PerarnauSerra}, Tao~\cite[Section~3]{Tao}, and Bedert~\cite[Section~2]{Bedert}.

There are two different types of overlap. As the Bohr set $B(v;\delta)$ is the union of $v$ intervals
\[
 I_{v,j}(\delta)=
 \{t\in\T:\|t-j/v\|\le\delta/v\},
\]
where $0\le j<v$. For each of the above intervals, we say it has \emph{rational center} $j/v$. 
Intervals belonging to different speeds may have the same rational center, which gives us \emph{Type I} overlaps. 

\begin{figure}[h]
    \centering
    \includegraphics[width=0.39\linewidth]{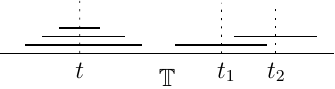}
    \caption{Type I overlaps (left) with rational center $t$ and Type II overlaps (right)}
    \label{fig: overlap}
\end{figure}

For example, the intervals with centers $1/2$ for speeds $2$ and $4$ are nested. More generally, if $a/q\in[0,1)$ is in lowest terms, it is a center for speed $v$ precisely when $q\mid v$. We will show that Type I overlaps can be controlled by the arithmetic of $V$. 

To study Type I overlaps, at each center, we will only keep the largest interval and remove all the others. This leaves the union unchanged. Let $M_\delta(t)$ count the remained intervals at $t$, and let
\[
 R_V(\delta)=\int_\T(N_\delta-M_\delta)\dd\mu
\]
be the total length removed, counted with multiplicity. Since the union is unchanged, $R_V(\delta_0)\le g(V)$. Thus it is enough to find a positive amount of overlap between intervals with a common center. This is the arithmetic case of our approach. 

There is another type of overlaps, that overlaps of intervals with distinct centers. They form \emph{Type II} overlaps, see Figure~\ref{fig: overlap}.

Let $Q_V(\delta)$ be the sum of the Type II intersection, which  measures of intervals with distinct centers over ordered pairs of distinct speeds. See  subsection~3.2 for the explicit definition. For distinct positive integers $u,v$, we define
\[
  H(u,v)=\frac{\max(u,v)}{\gcd(u,v)}.
\]
We call $H(u,v)$ the \emph{height} of the pair.

Lemma~\ref{lem: type ii} shows that the contribution of a pair $(u,v)$ lies between
\[
(2\delta)^2-\frac{4\delta}{H(u,v)} \quad\text{and}\quad(2\delta)^2.
\]
Thus, at $\delta_0$, if all but very less of the ordered pairs satisfy $H(u,v)>Cn$ for a sufficiently large constant $C$, then most pairs contribute close to $1/n^2$.

Type II intersections, in particular $Q_V$ alone do not give the required uncovered measure bound, because many intersecting pairs may all come from the same small region of overlap. Indeed, suppose exactly $r$ intervals of distinct centers cover a point $t$, they form $r(r-1)$ ordered Type II overlap pairs, but the number of intervals beyond the first is only $r-1$. 

\begin{figure}[h]
    \centering
    \includegraphics[width=0.5\linewidth]{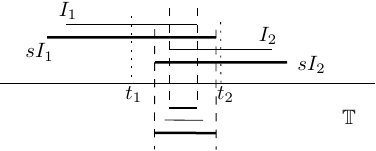}
    \caption{An example at level $0$: $I_1$ and $I_2$ are intervals with center $t_1$ and $t_2$. Dilate them by factor $s$ around their centers, the resulted overlaps exceeds linear growth (the middle line between the two overlaps at the bottom of the figure is the predicted linear growth size). The union of extra gain at any given level is captured by the spillover argument. } 
    \label{fig:Type II overlaps}
\end{figure}

Our new ingredient is a geometric argument that avoids this difficulty by comparing two radii.  Proposition~\ref{spill:pair spillover} proves that, if $0<\delta_0\le\delta_1<1/2$ and $s=\delta_1/\delta_0$, then
\[
 Q_V(\delta_1)\ge sQ_V(\delta_0)
       +2\bigl(w_V(\delta_1)-s w_V(\delta_0)\bigr).
\]
The inequality has a similar purpose to the \emph{spillover argument} for compact groups in the work of Tran and the second author~\cite[Lemma~9.4]{JingTran} and its geometric refinement in the work of Meng and the second author~\cite[Proposition~3.1]{JingMeng}: it is an argument to handle fiber and level sets together, and an initial overlap forces additional growth. Those results concern product sets and sumsets, in particular comparing $\mu(A+A)$ and $2\mu(A)$. Our objects is to compare the total Type II overlaps when each of the interval is dilated about it center. We prove the interval inequality used here directly in Section~\ref{sec: spillover}. As here for the intervals, the integrated overlap above any fixed multiplicity (or in other words, level sets) grows at least linearly when the intervals are enlarged about their centers. On the other hand, the contribution from intervals with a common center is exactly linear in the radius. These terms therefore cancel when we subtract the smaller radius pair count, multiplied by the ratio of the radii, from the larger radius pair count. This leaves a comparison involving only Type II overlaps.

We apply this inequality with $\delta_1=3\delta_0/2$. At the larger radius the total length of the Bohr sets is $3/2$, so $w_V(\delta_1)\ge1/2$. Type II pair estimates give an upper bound for $Q_V(\delta_1)$.
\[
Q_V(\delta_1)\le
\frac94\left(1-\frac1n\right),
\]
The dilation spillover inequality gives,
\[
Q_V(\delta_1)\ge
\frac32 Q_V(\delta_0)+1-3g(V),
\]
Substituting these bounds gives,
\[ 3g(V)\ge 1+\frac32 Q_V(\delta_0) -\frac94\left(1-\frac1n\right).  \] 

This geometric argument gives a uniform positive lower bound for $g(V)$ whenever the Type II contribution $Q_V(\delta_0)$ is sufficiently large. By the pairwise estimates in Lemma~\ref{lem: type ii}, this is the case when very few pairs satisfy $H(u,v)\le Cn$, since most pairs then behave almost independently after the intersections at common centers have been removed.

It remains to consider the complementary case, when Type I overlaps dominates. The arithmetic part of the proof does this when many pairs of speeds have a small ratio after canceling out their common factors. Theorem~\ref{mp:linear Type I} proves that, for every fixed $C\ge1$ and $\varepsilon>0$,
\[
 \#\{\{u,v\}\subset V:H(u,v)\le Cn\}\ge\varepsilon n^2
\]
implies that $R_V(\delta_0)\ge\kappa>0$ for some constant $\kappa$ depending only on $C$ and $\varepsilon$. Since $R_V(\delta_0)\le g(V)$, this gives the required uniform positive lower bound in the remaining case.

Let us briefly explain the arithmetic argument. 
We use the large GCD estimates of Green and Walker~\cite{GreenWalker} and refine their argument to produces a subset $U\subseteq V$ with $|U|$ a positive proportion of $n$, with elements can be written as
\[
 u=\xi\frac{a_u}{b_u},\qquad 1\le a_u\le A,\quad 1\le b_u\le B,\quad AB\ll|U|.
\]
The implied constant depends only on $C$ and $\varepsilon$. There are only $AB$ possible numerator and denominator pairs, so it places $U$ in a dense part of a small rectangular array of rational numbers. A matching argument on the GCD graph then finds many disjoint pairs with bounded reduced ratio. Each such pair gives overlap at common centers, and the matching lets us add these contributions back without overcounting. The proof mostly follows Green--Walker argument, and we modify the quantity optimized in the their argument by a factor that also records how many elements remains in $U$ after each valuation restricted to one prime at a time. This help us to keep a positive proportion of $V$ to make the resulting bound uniform in $n$.

\section{Removing intervals at rational centers and pair intersections}\label{sec:Type I}

As we mentioned in the introduction and section~2, there are two types of interval overlaps. Some pair of the intervals have the same rational center, and their intersection is simply the smaller interval. These overlaps are closely related to the \emph{height} of the speeds pairs. The remaining pairs have distinct centers, and they form the Type II overlaps. In this section, we study the behavior of the overlapping measures of two types when the radius changes.

\subsection{Type I overlaps}
 Recall that $N_\delta$ counts all the intervals in the sets $B(v;\delta)$, and $M_\delta$ counts only the largest interval at each rational center. Since the intervals with the same center are nested, removing the smaller ones preserves their union. The quantity $R_V(\delta)=\int_\T(N_\delta-M_\delta)\dd\mu$ records the total length of the intervals removed with multiplicity. 

If a center is represented by $a/q\in[0,1)$ in reduced lowest terms, then a speed has an interval centered there if and only if it is divisible by $q$. The remained interval then belongs to the smallest speed in $V$ divisible by $q$. For the center $0$, we use $a=0$ and $q=1$. Note that intervals belonging to a single speed have disjoint interiors, since $\delta<1/2$, and for simplicity, when we count the number of intervals at a point, we will ignore the finitely many endpoints. 

We first relate nested intervals removed at each center to the uncovered measure of interest. As discussed in the introduction, when the radius is $1/2n$, the sum of the interval lengths is exactly one. If these sets overlap, the sum counts some points more than once. Consequently, their union has measure less than 1, and the amount counted more than once is exactly the measure left uncovered. We record this as a simple lemma below. 

\begin{lemma}\label{lem: base identities}
Let $\delta_0=1/2n$. Then we have
\[
 \int_\T N_{\delta_0}\dd\mu=1,\qquad
 \int_\T(N_{\delta_0}-1)_+\dd\mu=g(V),\qquad
 0\le R_V(\delta_0)\le g(V).
\]
More generally, if $g_\delta$ denotes the measure of the complement of $\bigcup_{v\in V}B(v;\delta)$, then
\[
 \int_\T(N_\delta-1)_+\dd\mu=2n\delta-1+g_\delta.
\]
\end{lemma}
\begin{proof}
Each set $B(v;\delta)$ has measure $2\delta$. Also, we have the trivial identity $a=\mathbbm1_{a>0}+(a-1)_+$ for every nonnegative integer $a$. By applying this identity to $N_\delta$ and integrating, we obtain the general formula. At $\delta_0=1/2n$, $\int_\T N_{\delta_0}\dd \mu=2n \delta_0=1$, so the first two assertions follow. Finally, removing extra nested intervals preserves the union, and hence
\[
 1-g(V)=\mu\{M_{\delta_0}>0\}
       \le\int_\T M_{\delta_0}\dd\mu
       =1-R_V(\delta_0).
\]
This proves the last assertion.
\end{proof}

In the later proofs, we will often pass to a subset of the speeds and compare the two radii. The next lemma shows that both operations have a simple effect on the Type I loss.

\begin{lemma}
\label{lem: type I loss monotone}
For $U\subseteq V$ and $0<\delta<1/2$, define $N_{U,\delta}, M_{U,\delta},R_{U,\delta}$ similarly with $V$ replaced by $U$. Then we have pointwise inequality
\[
 N_{U,\delta}-M_{U,\delta}\le N_\delta-M_\delta. 
\]
 Moreover, $R_U(\delta)$ is linear in $\delta$ on this range.
\end{lemma}
\begin{proof}
Consider first that the intervals are at a fixed rational center. Their contribution to $N-M$ is the sum of their indicators minus the indicator of the largest interval. If we add another interval contained in the largest one, this contribution increases by the indicator of the new interval. If we add a larger interval, it increases by the indicator of the previously largest one. Adding the speeds of $V\setminus U$ one at a time and summing over the centers proves the pointwise inequality.

For the second part, we consider a center with reduced denominator $q$ that occurs among the speeds of $U$. Write $v_q=\min\{v\in U:q\mid v\}$. Its contribution to the integrated loss $R$ is
\[
 2\delta\Bigg(\sum_{v\in U,\, q\mid v}\frac1v
                    -\frac1{v_q}\Bigg).
\]
Observe that the expression in brackets does not depend on $\delta$. Summing over all centers proves the lemma.
\end{proof}

\subsection{Type II overlaps}

As we mentioned in the introduction, the intersection of two Bohr sets has two different types. Some pairs of intervals have the same center, and their intersection is simply the smaller interval. The remaining pairs have distinct centers, and they form the Type II overlaps. We separate Type II contributions because the first is exactly linear in $\delta$, and the second is close to the quadratic quantity $(2\delta)^2$ when the height $H(u,v)$ is large.

For distinct speeds $u,v$, let
\[
 q_\delta(u,v)=
 \sum_{\substack{0\le a<u,\ 0\le b<v\\a/u\ne b/v}}
    \mu\bigl(I_{u,a}(\delta)\cap I_{v,b}(\delta)\bigr).
\]
Thus $q_\delta(u,v)$ counts intersections of intervals with distinct centers. Define two quantities
\[
 Q_V(\delta)=\sum_{\substack{u,v\in V\\u\ne v}}q_\delta(u,v),
 \qquad
 P_V(\delta)=\int_\T N_\delta(N_\delta-1)\dd\mu.
\]
Let us remark that all sums over pairs of speeds in this subsection are ordered. In other words, $P_V(\delta)$ is the sum of the measures of $B(u;\delta)\cap B(v;\delta)$ over ordered distinct $u,v\in V$.
The following calculation is a variant of the rank two Bohr set estimates of Perarnau and Serra~\cite[Propositions~7]{PerarnauSerra}.

\begin{lemma}\label{lem: type ii}
Let $h=H(u,v)$ be the height. Then
\begin{equation}\label{eq: pair bounds}
 4\delta^2-\frac{4\delta}{h}
          \le q_\delta(u,v)\le4\delta^2.
\end{equation}
Moreover,
\begin{equation}\label{eq:common center linear}
 P_V(\delta)-Q_V(\delta) =2\delta\sum_{\substack{u,v\in V\\u\ne v}}
                   \frac{\gcd(u,v)}{\max(u,v)}.
\end{equation}
In particular, $Q_V(\delta)\le 4\delta^2n(n-1)$.
\end{lemma}
\begin{proof}
We first reduce the problem to coprime speeds. Let $d=\gcd(u,v)$ and let $\chi_d:\T\to\T$ the continuous surjective group homomorphism with $\chi_d: t\mapsto dt$. 
Observe that $\chi_d$ preserves normalized measure on $\T$ and pulls back $B(u/d;\delta)$ and $B(v/d;\delta)$ to $B(u;\delta)$ and $B(v;\delta)$, respectively. Each normalized center has $d$ lifts, and intervals at different lifts of the same center cannot intersect, as the spacing of the lifts is $1/d$ and the sum of the half lengths is $\delta/u+\delta/v<1/d$. Thus $\chi_d$ preserves both the total intersection measure and its contribution from equal centers. It also preserves $q_\delta(u,v)$. We may therefore assume $\gcd(u,v)=1$, in which case $h=\max(u,v)$.

Consider two intervals in $\R$ with lengths $2\delta/u$ and $2\delta/v$, and their centers have distance $x$. Let $F(x)$  be the length of their intersection. It is easy to check that the function $F$ is even, nonnegative, and nonincreasing on positive reals. 

To compute $q_\delta(u,v)$, we will lift the centers to $M/u$ and $N/v$, with $M,N\in\Z$ are integers, and write $j=vM-uN$. Their difference of centers is $j/uv$. Since $u$ and $v$ are coprime, all solutions of $(M,N)$ to $j=vM-uN$ with a fixed $j$ can be obtained by
\[
(M,N)\mapsto(M+u\ell,N+v\ell),
\]
where $\ell\in\Z$. Therefore, their intersections are integer translates of a single interval of length $F(j/(uv))$.

We claim that these intersection intervals for different values of $j$, have disjoint interiors after projection to $\T$. Indeed, at any real point there is at most one active lift of an interval of each speed, and this is because $\delta<1/2$. Equal rational centers give the term $j=0$, and the other possible differences for equal centers are nonzero integers, outside the support of $F$.
Therefore, we have
\[
 q_\delta(u,v)=2\sum_{j\ge1}F\left(\frac{j}{uv}\right).
\]

We now compare this sum with the integral of $F$. Since $F$ is
nonincreasing on the positive half line, its left and right
Riemann sums give
\[
 2\sum_{j\ge1}F\left(\frac{j}{uv}\right)
 \le uv\int_\R F(x)\dd x
 \le 2F(0)+2\sum_{j\ge1}F\left(\frac{j}{uv}\right).
\]
By Fubini's theorem, the integral of this intersection function
is the product of the two interval lengths. Thus
\[
 uv\int_\R F(x)\dd x=4\delta^2,\qquad\text{and}
 \qquad F(0)=\frac{2\delta}{\max(u,v)}=\frac{2\delta}{h}.
\]
This proves \eqref{eq: pair bounds}.

For general $u,v$, there are $\gcd(u,v)$ common centers, and the intersection at each one has length $2\delta/\max(u,v)$. After summing over ordered pairs, their total contribution is $P_V(\delta)-Q_V(\delta)$. This proves \eqref{eq:common center linear}. The final claim of the lemma follows by summing the upper bound in \eqref{eq: pair bounds}.
\end{proof}

\section{Dilation of interval multiplicities}\label{sec: spillover}

We now prove the geometric inequality needed to compare two radii. For a family of intervals contribute in $N$, the quantity $\int(N-k)_+$ measures the multiplicity beyond the first $k$ layers. In this section we will show that, when all lengths are multiplied by the same factor, this quantity $\int(N-k)_+$ grows at least linearly. We will first work on the real line $\R$ and then pass to circular arcs by taking many periods of their real lifts.

The first combinatorial lemma allows us to express $\min(N,k)$ using unions of intervals over $\R$.

\begin{lemma}\label{spill:coloring}
Let $I_1,\ldots,I_m$ be finitely many bounded intervals over $\R$, and let $N(x)=\sum_i\mathbbm1_{I_i}(x)$, and $k\ge1$ be an integer. There is a coloring of the intervals with $k$ colors such that exactly $\min\{N(x),k\}$ colors occur at almost every point $x$.
\end{lemma}

\begin{proof}
Note that end points of intervals partition the line into finitely many open cells, we denote this family of open cells by $\mathscr C$. Now we order intervals by their left end points, and if two intervals share a same left end point, the shorter interval has smaller label. 

We color intervals greedily by their order. When a new interval begins at their left end point $x$, let $C_x$ be the open cell in $\mathscr C$ with left end point $x$, and suppose there are $m$ colored intervals contain $C_x$ as their interiors. Now we relabel this $m$ active intervals by the order of their right end point, in decreasing order of their right endpoints. 

If $m<k$, we assign the new interval a color that is not used by these $m$ active intervals. If $m\geq k$, we assign the new interval the color of the $k$-th active interval under the new labels. 

We claim that every point $x$ other than end points of intervals, exactly $\min\{N(x),k\}$ many colors occur at $x$. Indeed, when a new interval begins at $x$, either it does not enter the first $k$ positions over points in $C_x$, or it replaces the previous $k$-th interval and receive its color. When intervals end, those with smaller right end points leave first so the claim is again preserved. 
\end{proof}

We remark that a direct torus version of Lemma~\ref{spill:coloring} is false. Indeed, three intervals with centers $0,1/3,2/3$ and length $1/2$ intersect pairwise on sets of positive measure but have no triple intersection. A coloring as in the lemma with two colors would have to give each pair different colors.

We next stay in $\R$ and keep the interval lengths fixed. As the centers move apart, the union in each color class has nondecreasing length. Maximizing over colorings therefore shows that 
\[
\int_{\mathbb R}\min\{N(x),k\}\dd x
\]
cannot decrease either. Rescaling then gives the desired dilation inequality.

\begin{lemma}\label{spill:dilation}
Let $I_i(s)=[c_i-sr_i,c_i+sr_i]\subseteq\R$ for $1\le i\le m$,
where $r_i>0$, and let $N_s=\sum_i\mathbbm1_{I_i(s)}$. Then for every integer $k\ge1$, we have
\[
 \frac1s\int_\R(N_s-k)_+
\]
is nondecreasing for $s>0$. The same assertion holds for a
finite family of circular arcs with fixed centers and
half-lengths $sr_i$, on any range of $s$ for which every arc
has length less than one, see the next lemma.
\end{lemma}
\begin{proof}
We first consider the real intervals $J_i(\lambda)=[\lambda c_i-r_i,\lambda c_i+r_i]$, that the centers are all dilated by $\lambda$. We claim that the length of their union is nondecreasing for $\lambda>0$. Indeed, Each endpoint $\lambda c_i\pm r_i$ is a linear function of $\lambda$. Two distinct such functions agree at most once, so the order of the endpoints changes only finitely many times. Between these values of $\lambda$, each connected component of $\bigcup_i J_i(\lambda)$ has left endpoint $\lambda c_i-r_i$ and right endpoint $\lambda c_j+r_j$ for fixed indices $(i,j)$. The indices remain fixed even though the endpoints move. We also conclude that
\[
 \lambda c_i-r_i\le\lambda c_j-r_j,
 \qquad
 \lambda c_i+r_i\le\lambda c_j+r_j.
\]
And adding this gives $c_i\le c_j$. The derivative of the component's length is therefore $c_j-c_i\ge0$. Summing over components, and using continuity, we prove the claim.

Now we define $\widetilde N_\lambda$ be the function that counts the number of $J_i(\lambda)$ at each points. Let 
\[
\mathscr C : J_i(\lambda) \to \{1,\dots,k\}
\]
be an arbitrary coloring of intervals $J_i(\lambda)$ with $k$ colors.  For each color class, take the union of its intervals. The sum of the lengths of these unions is at most 
\[
\int_{\R}\min\{\widetilde N_\lambda(x),k\}\dd x,
\]
since the number of colors presented at $x$ is at most both $\widetilde N_\lambda(x)$ and $k$. Lemma~\ref{spill:coloring} gives a coloring which attains the equality. Hence
\[
 \int_\R\min\{\widetilde N_\lambda,k\} \dd\mu_\R
 =\max_{\mathscr C}\sum_{a=1}^k
       \left|\bigcup_{i:\,\mathscr C(i)=a}J_i(\lambda)\right|.
\]
Each union of intervals in a color class being maximized is nondecreasing in $\lambda$, by the preceding claim. The left hand side is therefore nondecreasing as well. On the other hand,
$\int_\R\widetilde N_\lambda=\sum_i2r_i$ is a constant.
Subtracting shows that $\int_\R(\widetilde N_\lambda-k)_+$ is nonincreasing, as 
\[
(\widetilde N_\lambda-k)_+ + \min\{\widetilde N_\lambda,k\} = \widetilde N_\lambda.
\]
Finally, the change of variables $x=sy$ gives
\[
 \frac1s\int_\R(N_s-k)_+
       =\int_\R(\widetilde N_{1/s}-k)_+,
\]
which proves the desired conclusion. 
\end{proof}

We now project the lemma to $\T$ by lifting intervals to $\R$ and approximate them by finite subfamilies of intervals. 

\begin{lemma}\label{spill:dilation in T}
The same conclusion in Lemma~\ref{spill:dilation} holds for a finite family of intervals in $\T$, all with fixed centers and lengths $2sr_i$, on any range of $s$ such that every interval has length less than one.
\end{lemma}
\begin{proof}
For each intervals in $\T$, choose representative centers $c_i\in[0,1)$. Consider the real lifts, and the full family of real lifts has centers $c_i+j$, with $j\in\Z$. Let $N_s^{\mathrm{lift}}$ count the lifted intervals containing $x$. Clearly, this function is $1$-periodic.

To get finitely many intervals, for an integer $L\ge1$, remain only the lifts with $-L\le j\le L$, and denote the remained intervals containing by $N_{s,L}$. Let $b_s=\max_i sr_i<1/2$ the maximum half length of intervals. Except for points from the two boundary intervals
\[
 [-L-b_s,-L+b_s]\ \cup\ [L+1-b_s,L+1+b_s],
\]
the function $N_{s,L}$ agrees with $\mathbbm1_{[-L,L+1)}N_s^{\mathrm{lift}}$. This is because an omitted lift can reach the interval $[-L,L+1)$ only within distance $b_s$ of one of its endpoints.

Observe that the integer translates of any interval have disjoint interiors, so both $N_s^{\mathrm{lift}}$ and $N_{s,L}$ are at most $m$. The map
$x\mapsto(x-k)_+$ is $1$-Lipschitz and vanishes at zero. The boundary strips have total length $4b_s$, and hence by the above discussions
\[
 \left|\int_\R(N_{s,L}-k)_+ -(2L+1)\int_\T(N_s^{\mathrm{lift}}-k)_+\dd\mu\right| \le4mb_s.
\]
Now we apply Lemma~\ref{spill:dilation} to two values $s_1<s_2$, and 
\[
\frac{1}{s_1}\int_\R(N_{s_1,L}-k)_+ \leq \frac{1}{s_2}\int_\R(N_{s_2,L}-k)_+. 
\]
We then keep $s_1,s_2$ fixed, divide by the number $2L+1$ of periods and let $L\to\infty$. The boundary error tends to zero, giving the desired conclusion.
\end{proof}

Let us remark that the final colorings of the real lifts need not be periodic, this is the reason we pass to $\T$ through a limit, and the reason that torus version of Lemma~\ref{spill:coloring} is false but the torus version of Lemma~\ref{spill:dilation} holds.

We can now compare the contribution of distinct centers at two radii. The above dilation lemma in $\T$ controls every layer of $\int_{\mathbb T}(N_\delta-k)_+\dd\mu$. Summing these layers gives a bound for the total pair intersections, and the Type I contribution cancels because it is exactly linear in the radius.

\begin{proposition}[Dilation spillover inequality]
\label{spill:pair spillover}
Let $0<\delta_0\le\delta_1<1/2$, and let $s=\delta_1/\delta_0$. We write $w_j=\int_\T(N_{\delta_j}-1)_+\dd\mu$. Then
\begin{equation}\label{eq:distinct center dilation}
 Q_V(\delta_1)\ge sQ_V(\delta_0)+2(w_1-sw_0).
\end{equation}
Moreover, if $\bigcup_{v\in V}B(v;\delta_0)=\T$, the final term $2(w_1-sw_0)$ is
$2(s-1)$.
\end{proposition}
\begin{proof}
For every integer $0\le a\le n$, we have the following identity, 
\[
 a(a-1)-2(a-1)_+=2\sum_{k=2}^{n-1}(a-k)_+.
\] 
Apply Lemma~\ref{spill:dilation in T} to all the original intervals $I_{v,j}(\delta_0)$ and sum over $k$. The intervals corresponding to any one speed in $V$ are disjoint. Since there are $n$ speeds, $N_\delta(t)\le n$ for every $t\in\T$. The identity above therefore gives
\[
 P_V(\delta_1)-2w_1
       \ge s\bigl(P_V(\delta_0)-2w_0\bigr).
\]
By \eqref{eq:common center linear}, the difference $P_V-Q_V$ is
linear in $\delta$. Subtracting its contribution from both
sides gives \eqref{eq:distinct center dilation}.

If the Bohr sets cover $\T$ at $\delta_0$, they also cover it at $\delta_1$. Lemma~\ref{lem: base identities} then gives $w_j=2n\delta_j-1$, and hence $w_1-sw_0=s-1$.
\end{proof}

\section{Rational boxes and Type I loss}\label{sec:rational-boxes}

We first consider a family of speeds described by fractions with bounded numerators and denominators in the following way. For a positive rational $\xi$ and positive integers $A,B$, we call
\[
 \{\xi a/b:a,b\in\N,\ 1\le a\le A,\ 1\le b\le B\}
\]
a \emph{rational box}, and call $AB$ its \emph{area}, and $A,B$ are its numerator and denominator \emph{side lengths}. Here the word area refers to the number of possible numerator denominator pairs. We only consider subsets of this box consisting of distinct positive integer speeds in the section. We will show that if the number of possible fractions is comparable to the size of the family, then many disjoint pairs have bounded reduced ratio. The overlaps at their common centers give a positive lower bound for the Type I loss.

\begin{lemma}\label{lem:dense rational box}
For every $C\ge1$ there is a constant $\kappa(C)>0$ with the following property. Suppose that $U$ consists of $m\ge2$ distinct positive integer speeds of $V$, and for some positive rational $\xi$, its members have representations $v=\xi a_v/b_v$ which lives in a rational box of numerator and denominator side lengths $A,B$, and area $AB\leq Cm$. 
Then we have
\[
 R_U(1/2m)\ge\kappa(C),\qquad\text{and}
 \qquad g(U)\ge\kappa(C).
\]
\end{lemma}

\begin{proof}
Throughout this proof, $R_U,N_U,M_U$ are evaluated at the radius $1/2m$. The first step is to find a constant depending only on $C$, such that for sufficiently large $m$, any subset whose pairwise reduced ratio heights greater than this constant contains at most half of $U$.

Let $\mathcal P$ consists of primes up to $X=\lceil\exp(8C)\rceil$, and hence
\[
 d_{\mathcal P}=\prod_{p\in\mathcal P}(1-1/p)\le\frac{1}{8C}.
\]
Let
\[
A_{\mathcal P}=\prod_{p\in\mathcal P}(1-p^{-1/2})^{-1},\quad\text{and}\quad S=\left\lceil(8CA_{\mathcal P})^2\right\rceil. 
\]

For a positive integer $a$, write $a=s(a)r(a)$, where all prime factors of $s(a)$ lie in $\mathcal P$, and no prime in $\mathcal P$ divides $r(a)$. Thus $s(a)$ is the $\mathcal P$-smooth part of $a$. We first discard representations that the $\mathcal P$-smooth part is large. Observe that he smooth-part tail satisfies
\[
 \sum_{s>S,\ s\ \mathcal P\text{-smooth}}\frac1s
 \le\frac1{\sqrt S}
       \sum_{s\ \mathcal P\text{-smooth}}\frac1{\sqrt s}
 =\frac{A_{\mathcal P}}{\sqrt S}\le\frac1{8C}.
\]
This implies that at most $AB/(8C)\le m/8$ pairs $(a,b)$ in the box $[1,A]\times[1,B]$ have $s(a)>S$, as the number is at most $B\sum_{s>S}\lfloor A/s\rfloor$. Similarly, at most $m/8$ have
$s(b)>S$. 

Now consider a subset $U_0$ of $U$ such that any pair of elements in $U_0$ has reduced ratio height greater than $S^2$. Now we discard elements $u=\xi a_u/b_u$ in $U_0$ that at least one element in its numerator denominator pair $(a_u,b_u)$ having $\mathcal P$-smooth part greater than $S$. So we discard at most $m/4$ elements. For the remaining pairs $(a,b)$, we claim that the pairs $(r(a),r(b))$ are distinct. Indeed, if two distinct speeds $u_1,u_2$ in $U_0$ with $(a_{u_1},b_{u_1})$ and $(a_{u_2},b_{u_2})$ give the same rough pair, their ratio is
\[
\frac{u_1}{u_2} = \frac{s(a_{u_1})s(b_{u_2})}{s(a_{u_2})s(b_{u_1})}
\]
and the reduced numerator and denominator are both at most $S^2$.

To bound $|U_0|$, it remains to count these distinct pairs. Without loss of generality we assume $A\ge B$. As the $m$ speeds in $U$ require $m$ distinct representing pairs, hence $AB\geq m$ and $A\ge m^{1/2}$. By inclusion-exclusion principle, at most $A d_{\mathcal P}+2^{|\mathcal P|}$ integers in $[1,A]$ are coprime to every prime in $\mathcal P$. Indeed, the main term is clearly $A d_{\mathcal P}$, and each of the $2^{|\mathcal P|}$ floor terms contributes an error at most one. There are at most $B$ possible values of $r(b)$. Hence the number of remained pairs is at most
\[
 (A d_{\mathcal P}+2^{|\mathcal P|})B \le\frac m8+\frac{Cm2^{|\mathcal P|}}{A} \le\frac m8+C2^{|\mathcal P|} m^{\frac12}.
\]
Now let $N_0 = \left\lceil(8C2^{|\mathcal P|})^2\right\rceil$. 
For $m\ge N_0$, the right hand side bound in the above inequality is at most $m/4$. Together with two discarded class with at most $m/4$ speeds, we conclude that any subset of $U$ with pairwise reduced ratio heights greater than $S^2$ has size at most $m/2$, as claimed.

We now turn this count into overlap estimates of Type I. Construct a graph $G$ with vertex set $V(G)=U$ and we join two speeds by an edge when their reduced ratio height is at most $S^2$. Let $\mathscr M$ be a maximal matching of $G$, and write $U' = U\setminus V(\mathscr M)$. Clearly $U'$ is an independent set in $G$. Using the bound in the the preceding claim, when $m>N_0$, the independent set has size at most $m/2$, which implies that $\mathscr M$ contains at least $m/4$ edges. At the radius level $1/2m$, the intersections at the common centers of each
edge have total measure at least $1/(mS^2)$.

We claim that summing these measures provides a lower bound on the Type I loss. To see this, fix a time $t$ and a rational center, and suppose that $k$ of the original intervals at that center contain the time $t$. Type I overlaps removes $k-1$ of their contributions when $k\ge1$. On the other hand, each speed belongs to at most one pair in the matching. Thus the $k$ intervals can create for at most $\lfloor k/2\rfloor$ in pairs, which are edges in the matching. As $\lfloor k/2\rfloor\le(k-1)_+$, summing over centers and integrating gives $R_U\ge1/(4S^2)$.

For the small case that $2\le m<N_0$, one pair is actually enough. For $(u,w)$, its ratio is $a_vb_w/(a_wb_v)$, so its reduced ratio height is at most $AB\le Cm$. The intersections at its common centers therefore have total measure at least $1/(Cm^2)\ge1/(C N_0^2)$ and are bounded above by $R_U$ in the same way.

Now we define
\begin{equation}\label{eq: value of kappa}
    \kappa(C)=\min\{1/(4S^2),\,1/(C N_0^2)\}.
\end{equation}
Note that $\kappa$ only depends on $C$. This proves the first part of the lemma that 
\[
R_U\big(\frac{1}{2m}\big) \geq \kappa(C). 
\]

Finally, at the baseline $\int N_U=1$, and removing Type I overlaps preserves the Bohr union, as the uncovered measure equals the length removed at common centers (Type I), plus the overlap remaining among the remained intervals (Type II), with both counted with multiplicity. Therefore,
\[
 g(U)=R_U+\int(M_U-1)_+\ge R_U,
\]
which proves the second part of the lemma.
\end{proof}

We remark that the proof and \eqref{eq: value of kappa} gives an explicit choice of $\kappa(C)$.

The preceding lemma treats a whole family inside one rational box. We next record a lemma proves a similar conclusion when the set is inside a family of rational boxes.
More generally, the summation of Type I overlaps losses from disjoint subsets is at most the Type I overlaps loss of their union.

\begin{lemma}\label{lem:Type I superadditive}
Let $U_1,\ldots,U_k$ be pairwise disjoint subfamilies of $V$ of distinct positive integer speeds. For $0<\delta<1/2$,
\[
 R_V(\delta)\ge\sum_{j=1}^kR_{U_j}(\delta).
\]
\end{lemma}
\begin{proof}
Fix a rational center. Let $N_j(t)$ count the intervals from $U_j$ at this center that contain $t$, and let $M_j$ be the indicator function of their union. Both are zero if there is no such interval. The intervals have the same center and are therefore nested. The loss from Type I overlaps from all these intervals is pointwisely
\[
 \sum_jN_j-\max_jM_j\ge\sum_j(N_j-M_j).
\]
Adding the speeds outside $\bigcup_jU_j$ can only increase this loss. Summing over centers and integrating proves the lemma. 
\end{proof}

Hence by applying Lemma~\ref{lem:dense rational box} to $U_j$ at its own radius $1/(2m_j)$. Type I loss is exactly linear in the radius by Lemma~\ref{lem: type I loss monotone}, so
\[
 R_{U_j}\big(\frac{1}{2n}\big)=\frac{m_j}{n}R_{U_j}\big(\frac{1}{2m_j}\big)
 \ge\frac{m_j}{n}\kappa(C_0)
\]
whenever each $U_j$ is contained in a rational box of area at most $C_0 m_j$.  Now sum over $j$ gives
\[
R_V\big(\frac{1}{2n}\big)\ge\kappa(C_0)  \frac{\sum_jm_j}{n}.
\]
In particular, a positive proportion of the speeds in one rational box gives a fixed constant improvement in the lonely runner bound, regardless of the remaining speeds.

\begin{corollary}[A positive proportion of speeds in a rational box]
\label{cor:rational-box-gain}
Let $|V|=n$, and let $U\subseteq V$ have $m\ge2$ speeds all lying in a rational box of area at most $Cm$. Then
\[
 g(V)\ge R_V\big(\frac{1}{2n}\big)\ge\frac mn\kappa(C).
\]
Consequently
\[
 \ML(V)>
 \frac{1+(m/n)\kappa(C)/2}{2n}.
\]
The other speeds in $V\setminus U$ are arbitrary.
\end{corollary}
\begin{proof}
The lower bound for $R_V$ follows from Lemma~\ref{lem:Type I superadditive} with one subset. The inequality $g(V)\ge R_V(1/2n)$ is Lemma~\ref{lem: base identities}. Write $\alpha=(m/n)\kappa(C)$, when the radius increases from $1/2n$ to $(1+\alpha/2)/(2n)$, the sum of the measures of the Bohr sets increases by $\alpha/2$. Their union can increase by no more than this amount. Its complement therefore still has measure at least $\alpha/2>0$, which gives the strict inequality for $\ML(V)$.
\end{proof}

\section{Valuation concentration and the large GCD estimate}
\label{sec:large GCD}

In this section, we recall two results of Green and Walker~\cite{GreenWalker}. The first controls the exponents of a prime in a collection of pairs of integers, and the second bounds the sizes of two sets when many pairs have a large common divisor. These are the two inputs used for finding the rational box in the next section.

The first result is a concentration estimate for a probability measure on $\mathbb Z^2$. We state it in terms of probability vectors, as this is the form used below.

\begin{lemma}[Green--Walker]\label{gwself:concentration}
For every $2<Q<3$, there is a constant $C_Q>0$ with the following property. Let $(\alpha_i)_{i\in\mathbb Z}$ and $(\beta_j)_{j\in\mathbb Z}$ be finitely supported sequences of nonnegative numbers with
\[
 \sum_i\alpha_i=\sum_j\beta_j=1.
\]
Suppose that a finitely supported probability measure $\mu$ on $\mathbb Z^2$ satisfies
\[
 \mu(i,j)\le c\lambda^{|i-j|}
                 (\alpha_i\beta_j)^{1-1/Q},
\]
where $0<c\le1$ and $0<\lambda\le4/5$. Then $c\ge1/9$, and there is an integer $k$ such that
\[
 \sum_{|i-k|+|j-k|\ge2}\mu(i,j)   \le C_Q\lambda^{2Q-2}.
\]
If both sequences are supported on nonnegative indices, then $k$ can be chosen nonnegative.
\end{lemma}

This is exactly \cite[Lemma~2.1]{GreenWalker}. To match their notation, we may choose $t=Q/(Q-1)$, $x_i=\alpha_i^{1/t}$, and $y_j=\beta_j^{1/t}$. Then $\sum_i x_i^t=\sum_j y_j^t=1$. Their parameter is $\varepsilon=Q-2$, so their exponent $Q+\varepsilon$ is $2Q-2$. For the final claim, we may replace $k$ by $\max\{k,0\}$ if necessary, as this does not increase $|i-k|+|j-k|$ on the support of $\mu$.

Let us now fix the notation. Let $X,Y$ be finite nonempty sets of positive integers, and let $D>0$. Define a \emph{GCD graph} $G$, which is a bipartite graph with vertex set $X\cup Y$, and two vertices $(x,y)$ are adjacent when $\gcd(x,y)>D$. We define the \emph{edge density} of $G$ to be
\[
 \rho=\frac{|E(X,Y)|}{|X||Y|}.
\]

We next record the large GCD estimate.  The following form follows from their theorem by an elementary reduction.

\begin{theorem}[Green--Walker]\label{gwself:gcd}
Let $q_0>2$. There is an absolute constant $K_0(q_0)\ge1$ with the following property. Let $A,B\ge1/2$ and $D>0$, and let $X\subseteq[A,2A]$ and $Y\subseteq[B,2B]$ be finite nonempty sets of positive integers. Suppose that a nonempty collection $E\subseteq X\times Y$ has density $\rho$ and satisfies $\gcd(x,y)>D$ for every $(x,y)\in E$. Then
\[
 \rho^{q_0}|X||Y|\frac{D^2}{AB}\le K_0.
\]
\end{theorem}

\begin{proof}
By \cite[Theorem~1.1]{GreenWalker}, with $\varepsilon=q_0-2$, there is an absolute constant $C\ge1$ such that
\[
 |X||Y|\le C\rho^{-q_0}\frac{AB}{D^2}
\]
when $A,B,D\ge1$ and $D\le\min\{A,B\}$. Indeed, the number of small primes in their statement is bounded in terms of $\varepsilon$ alone. We show that $K_0=4C$ works on the larger parameter range.

Suppose $D<1$. As $|X|\le2A$ and $|Y|\le2B$, so
\[
 \rho^{q_0}|X||Y|\frac{D^2}{AB}\le4\le K_0.
\]
We may therefore assume that $D\ge1$. Since $E$ is nonempty, we must also have $A,B\ge1$. For example, if $A<1$, then $X=\{1\}$, contradicting $\gcd(x,y)>D\ge1$ on any pair in $E$.

Set $D'=\min\{D,A,B\}$. A pair $(x,y)\in E$ gives
\[
 D<\gcd(x,y)\le\min\{x,y\}\le2\min\{A,B\},
\]
so $D'\ge D/2$. The parameters $A,B,D'$ satisfy the hypotheses of Green--Walker's theorem. Applying it with $D'$ in place of $D$, we obtain
\[
 |X||Y|\le C\rho^{-q_0}\frac{AB}{(D')^2}
          \le4C\rho^{-q_0}\frac{AB}{D^2},
\]
as required.
\end{proof}

We will use $K_0=K_0(9/4)$ in later proofs for the absolute constant in the Green--Walker theorem and $C_Q$ for the constant in Lemma~\ref{gwself:concentration}.

\section{Extracting a large rational box}\label{sec:linear ratio sparsity}

The large GCD bound in Section~\ref{sec:large GCD} controls the size of a family with many pairs having a large common divisor. We need a little more: a positive proportion of the speeds must lie in a rational box with its area is comparable to its size. The difficulty is that repeated restriction to valuation classes may leave very few speeds. We address this issue in the section by slightly modify Green--Walker's argument \cite{GreenWalker} and include the remained cardinality in the quantity we maximize.  

Throughout the section, we will use the two estimates proved in Section~\ref{sec:large GCD}. Theorem~\ref{gwself:gcd}, with $q_0=9/4$ and Lemma~\ref{gwself:concentration} with an exponent $Q$ different from $q_0$, and remain the notation $C_Q$ for its constant. We will apply Lemma~\ref{gwself:concentration}  on GCD graph defined on $X\times Y$, after partition vertices into $p$-adic valuation slices for some prime $p$. Here for a prime $p$, define the \emph{$p$-adic valuation} $\nu_p(x)$ to be the exponent of $p$ in $x$ where $x\in\Z$, and extend this notation to positive rationals by
\[
\nu_p\big(\frac xy\big)=\nu_p(x)-\nu_p(y).
\]
By a $p$-adic valuation slice of $X$ we meant
\[
X_i=\{x\in X:\nu_p(x)=i\}. 
\]

Let us fix the choices of parameters. The exponent $q>q_0$ so the large GCD bound will control the remained edge density. The positive exponent $\sigma$ penalizes a loss of cardinality.
We choose them so that the errors from the concentration lemma are summable over the primes. Here we set
\[
 q=5/2,\qquad \sigma=1/20,\qquad Q=\frac q{1+\sigma}=50/21,
 \qquad\zeta=\frac{2Q-2}{q}=116/105>1.
\]
Choose an integer $p_0$ sufficiently large that
\[
 C_Q\sum_{p>p_0}p^{-\zeta}\le1/2.
\]
For example, it suffices to take
\[
 p_0=\left\lceil\max\left\{2,
       \left(\frac{2C_Q}{\zeta-1}\right)^{1/(\zeta-1)}
                  \right\}\right\rceil.
\]
Let $K=1000$. The choice of these parameters is not unique and we did not optimize them in this paper, the point is that all these constants are absolute.

\subsection{Extraction with a cardinality penalty}

The next proposition makes the rational box extraction quantitative. Its proof modifies the prime restriction in \cite[Section~3]{GreenWalker} by maximizing a quantity that also records how many elements remain. Once the valuations concentrate, the counting argument of \cite[Lemma~4.1]{GreenWalker} produces the required rational box.

\begin{proposition}\label{mp:extraction}
Let $X_0\subset[A_0,2A_0]$, $Y_0\subset[B_0,2B_0]$ be finite nonempty sets of positive integers, and let $E_0\subset X_0\times Y_0$ have density $\rho_0>0$ and satisfy $\gcd(x,y)>D_0>0$ for every edge. Suppose
\[
 F_0:=\rho_0^q|X_0||Y_0|\frac{D_0^2}{A_0B_0}K^{-\pi(p_0)}>0.
\]
There is a subset $U\subset Y_0$ such that
\[
 |U|\ge \frac14\left(\frac{F_0}{K_0}\right)^{24}|Y_0|
\]
and, for some positive rational $\xi$ and positive integers
$A_U,B_U$, every $u\in U$ has a representation $u=\xi a_u/b_u$ where $a_u\in [1,A_U]\cap\N$ and $b_u\in [1,B_U]\cap \N$, and
\[
 A_UB_U\le\frac{512}{F_0}|U|.
\]
\end{proposition}
\begin{proof}
We will apply a finite sequence of operations, starting with $A_0,B_0,E_0,X_0,Y_0,D_0$. Suppose at the current step, we have $A,B,E,X,Y,D$. Choose any prime $p$ dividing an element of the  set $X\cup Y$, and choose $p$-adic valuation slices $X_i,Y_j$ containing at least one edge. We first replace the two sets $X,Y$ by $p^{-i}X_i$ and $p^{-j}Y_j$. Keep the corresponding edges and replace
\[
 (A,B,D)\quad\text{by}\quad(A/p^i,B/p^j,D/p^{\min(i,j)}).
\]
The common divisor condition is preserved, and $p$ no longer divides any element of the new sets. As in the above operations, new prime appears, so there are only finitely many possible steps. 
Observe also that each resulting set is obtained from a subset of $X_0$ or $Y_0$ by one common dilation, and their cardinalities can only decrease.

For any such operation, write $a=|X|$, $b=|Y|$, and $a_0=|X_0|$, $b_0=|Y_0|$ for the starting sets, and let $\rho$ be the edge density on the GCD graph defined on $X\times Y$ with respect to $D$.
Let $z=ab/(a_0b_0)$ be the proportion of pairs remained from the initial sets, and let $r$ count the number of primes at most $p_0$ dividing an element of $X\cup Y$. 
Now we maximize
\[
 F=\rho^q ab\frac{D^2}{AB}\,z^\sigma K^{-r}
\]
over the all possible sequence of operations of finite lengths, and fix a maximizer $F_{\text{max}}$. The new factor $z^\sigma$ in our proof prevents us from improving the other factors at an uncontrolled cost in cardinality. At the initial step, $z=1$ and $r\le \pi(p_0)$, so the maximum is at least $F_0$.
On the other hand, Theorem~\ref{gwself:gcd} gives
\[
 F_{\text{max}}\le K_0\rho^{q-q_0}z^\sigma
       =K_0\rho^{1/4}z^{1/20}.
\]
Since $0<\rho,z\le1$, it follows  that
\begin{equation}\label{mp:massdensity}
 \rho\ge(F_0/K_0)^4,
 \qquad\text{and}\qquad z\ge(F_0/K_0)^{20}.
\end{equation}
In particular, we have $F_0/K_0\le1$.

We now apply the valuation argument to the maximizing sets. Fix a prime $p$ dividing one of their elements. Let $\alpha_i=|X_i|/a$, $\beta_j=|Y_j|/b$ be the densities of the $p$-adic valuation slices, and let $\mu(i,j)$ be the fraction of edges in $X_i\times Y_j$, that is
\[
\mu(i,j)=\frac{|E(G)\cap(X_i\times Y_j)|}{|E(G)|}.
\]
For nonzero $\mu(i,j)$, by restricting to induced subgraphs on these classes gives edge density $\rho_{ij}=\rho\mu(i,j)/(\alpha_i\beta_j)$ and remained
cardinality ratio $z_{ij}=z\alpha_i\beta_j$. Also
\[
 \frac{(D/p^{\min(i,j)})^2}{(A/p^i)(B/p^j)}
       =\frac{D^2}{AB}p^{|i-j|}.
\]
Let $F_{ij}$ be the quantity $F$ restricted on $(X_i,Y_j)$. Hence the maximality of $F_{\text{max}}$  gives
\[
 \frac{F_{ij}}{F_{\text{max}}}
 =\mu(i,j)^q(\alpha_i\beta_j)^{1+\sigma-q}
      p^{|i-j|}K^{r-r_{ij}}\le1,
\]
where $r_{ij}$ is defined similarly, that is the number of primes at most $p_0$ dividing an element of $X_i\cup Y_j$. 

Observe that if $p<p_0$, at least the prime $p$ is removed in $r_{ij}$ and hence $r_{ij}\leq r -1$. If $p\geq p_0$, the number of primes cannot increase and $r_{ij}\leq r $ trivially. In either cases we have
\[
r-r_{ij}\ge\mathbbm1_{p\le p_0}. 
\]
Therefore we have
\[
 \mu(i,j)\le
 K^{-\mathbbm1_{p\le p_0}/q}
 (\alpha_i\beta_j)^{(q-1-\sigma)/q}p^{-|i-j|/q}.
\]
The exponent on $\alpha_i\beta_j$ is exactly $1-1/Q$. Now we apply Lemma~\ref{gwself:concentration} with exponent $Q$ and $\lambda=p^{-1/q}\le2^{-2/5}<4/5$. 

If the prime $p<p_0$, by the choice of parameters we will have
\[
K^{-1/q} < \frac19
\]
contradicting the size restriction on coefficients in Lemma~\ref{gwself:concentration}. Thus no prime at most $p_0$ divides any element of the maximizing
sets.

For each prime $p>p_0$ dividing an element of $A\cup B$, Lemma~\ref{gwself:concentration} gives an integer $k_p\ge0$ such that all but at most $C_Qp^{-\zeta}|E|$ pairs
$(u,v)\in E$ satisfy
\begin{equation}\label{eq: square-free}
 |\nu_p(u)-k_p|+|\nu_p(v)-k_p|\le1.
\end{equation}
The sum of the discarded proportions is at most $1/2$, by our choice of $p_0$, so at least half of edges in $E$ remains. Let $N=\prod_p p^{k_p}$, with the product over the primes dividing the maximizing sets, and let $E'\subseteq E$ be the set of remaining edges and so $|E'|\geq |E|/2$. 

Now for every vertex incident to an edge in $E'$, write $v=Nv_1/v_2$ in lowest terms. The prime valuation condition \eqref{eq: square-free} shows that $v_1,v_2$ are square-free. Put $v'=v_1v_2$. On every edge $(u,v)\in E'$,
\begin{equation}\label{eq: good edges}
 u'v' =\frac{uv}{\gcd(u,v)^2}\le\frac{4AB}{D^2}=:4Q_0.
\end{equation}
For vertices $u\in[A,2A]$ incident to edges in $E'$ and satisfying $u'\le L$, one has
\[
 u_1\le \Big(\frac{2AL}{N}\Big)^{\frac12},\qquad\text{and}\qquad u_2\le \Big(\frac{NL}{A}\Big)^{\frac12}.
\]
As this representation is injective, we conclude that there are at most $\sqrt2L\le2L$ such vertices.

Next we show that if $u'$ is large, we are able to find many neighbors $v$ of $u$. This will force their products $v'$, and hence their numerators and denominators to be small. Let $G'$ be the subgraph of the GCD graph $G$ defined on the maximizing sets, such that $E(G')=E'$ and $V(G')=V(G)$. The edge density of $G'$ is at least  $\rho/2$. The degree count from Theorem~\ref{gwself:gcd} shows that at least $\rho a/4$ vertices of $X$ have at least $\rho b/4$ neighbors. Call this high degree subset of $X$ by $X'$ and hence $|X'|\geq \rho a/4$. 

Set $L=\rho a/16$. By the computation above, at most $2L=\rho a/8$ vertices $u\in X$ have $u'\le L$, thus at least $\rho a/8$ vertices in $X'$ satisfy $u'>L$. Fix such a vertex $u$, and let $U'$ be its neighbors. Then for every $v\in U'$ by \eqref{eq: good edges}
\[
 m:=|U'|\ge \frac{\rho b}{4},\qquad\text{and}
 \qquad v'\le\frac{64Q_0}{\rho a}.
\]

As every $v\in U'\subseteq Y$ satisfies $v\in [B, 2B]$, and $v=Nv_1/v_2$ in the reduced form. As a consequence, 
\[
v_1^2 = \frac{v}{N} v' \leq \frac{2^7 Q_0 B}{N\rho a},\qquad\text{and}\qquad v_2^2 = \frac{N}{v}v' \leq \frac{2^6NQ_0}{B\rho a}.
\]
Thus $U'$ is contained in a rational box
\[
\Big\{ N\frac{c}{d} : c,d\in\N \text{ and } c\leq \Big(\frac{2^7 Q_0 B}{N\rho a}\Big)^{\frac12}, d\leq \Big(\frac{2^6NQ_0}{B\rho a}\Big)^{\frac12} \Big\}.
\]
Its area is at most
\[
 2^{\frac{13}{2}}\frac{Q_0}{\rho a}
 \le 2^{\frac{17}{2}}\frac{Q_0}{\rho^2ab}\,m
 \le\frac{2^9}{F_0}m.
\]
The final inequality follows from $F_{\text{max}}\ge F_0$ and $Q_0/(\rho^2ab)=\rho^{q-2}z^\sigma K^{-r}/F_{\text{max}}\le1/F_0$.
Finally, we take integer floors of the two coordinate bounds to obtain positive integers $A_U,B_U$. They are positive because $U'$ is nonempty, and their product satisfies the same area bound.

Recall that $Y$ is obtained from $Y_0$ via dilations. Reverse the common dilation gives a subset $U$ of $Y$ and this changes only the common factor of the rational box.
Finally, $a/a_0\le1$ implies $b/b_0\ge z$. Hence
\eqref{mp:massdensity} gives
\[
 |U|=m\ge\frac{\rho z b_0}{4} \ge\frac14(F_0/K_0)^{24}|Y_0|,
\]
as required.
\end{proof}

\subsection{Pairs with reduced ratio of order $n$}

The extraction proposition in the previous subsection concerns two fixed dyadic intervals. To apply it to arbitrary speeds, we partition them into such intervals. There could be many intervals, each containing very few speeds, so a direct pigeonhole argument would lose the desired positive proportion. The large GCD estimate shows that these sets with few speeds account for few pairs of bounded reduced ratio. We may therefore find the required pair of intervals among the sets with many speeds.

\begin{theorem}[Many bounded reduced ratios force Type I loss]
\label{mp:linear Type I}
For every $C\ge1$ and $\varepsilon>0$, there is a constant $\kappa(C,\varepsilon)>0$ with the following property. If $V$ consists of $n\ge2$ distinct positive integers and has at least $\varepsilon n^2$ unordered pairs of reduced ratio height at most $Cn$, then
\[
  R_V\big(\frac{1}{2n}\big)\ge\kappa(C,\varepsilon).
\]
\end{theorem}
\begin{proof}
Let us partition $V$ into the sets $V_j=V\cap[2^j,2^{j+1})$ and write $a_j=|V_j|/n$. Note that $\sum_ja_j=1$. Let $e_{jk}$ count the number of ordered pairs of distinct speeds in $V_j\times V_k$ with reduced ratio height at most $Cn$. Thus every such pair satisfies
\[
 \gcd(u,v)\ge\frac{\max(u,v)}{Cn}
       \ge\frac{\max(2^j,2^k)}{Cn}
       >D_{jk}:=\frac{\max(2^j,2^k)}{2Cn},
\]
and we use $D_{jk}$ for the threshold. 
Now we apply the Green--Walker bound in Theorem~\ref{gwself:gcd} with $q_0=9/4$. For $r_0=1-1/q_0=5/9$, and write $L_C=(4K_0C^2)^{1/q_0}$, it gives
\begin{equation*}
 \frac{e_{jk}}{n^2}\le
 L_C\,2^{-|j-k|/q_0}(a_ja_k)^{r_0}. 
\end{equation*}
Note that the inequality is also valid if one set or the edge set is empty. Since $a_j,a_k\le1$, this gives $e_{jk}\le L_C2^{-|j-k|/q_0}n^2$. Thus pairs of dyadic intervals with widely separated indices, that is $|j-k|$ large, contain few pairs of speeds of reduced ratio height at most $Cn$.

Let $K_\ell=2^{-|\ell|/q_0}$ and let $B_0=\sum_{\ell\in\mathbb Z}K_\ell<\infty$. Set
\[
 \eta=\min\Big\{\frac12, \Big(\frac{\varepsilon}{2L_CB_0}\Big)^{18}\Big\}.
\]
Then $0<\eta<1$ and $2L_CB_0\eta^{1/18}\le\varepsilon$. We call $V_j$ small if $a_j<\eta$, and otherwise $V_j$ is large. Let $f_j=a_j^{r_0}$ and $f_j^{\rm sm}=f_j\mathbbm1_{a_j<\eta}$. By the choice of parameter, $2r_0=10/9$, and hence
\[
 \|f\|_2^2=\sum_ja_j^{10/9}\le1,
 \qquad\text{and}\qquad 
 \|f^{\rm sm}\|_2^2\le\eta^{1/9}\sum_ja_j\le\eta^{1/9}.
\]
Now we write $(K*f)_j=\sum_\ell K_\ell f_{j-\ell}$. Then by Cauchy--Schwarz and the convolution inequality we have 
\[
\|K*f\|_2\le\sum_\ell K_\ell\|f\|_2=B_0\|f\|_2.
\]
The number of ordered pairs counted by the $e_{jk}$ with at least one speed in a small set is at most
\[
 2L_C n^2\sum_jf_j^{\rm sm}(K*f)_j
 \le2L_CB_0\eta^{1/18} n^2\le\varepsilon n^2.
\]
The above convolution bound is the triangle inequality applied to the sequences $f$.

Each unordered pair gives two ordered pairs since its speeds are distinct. Thus there are at least $2\varepsilon n^2$ ordered pairs of reduced ratio height at most $Cn$. At least $\varepsilon n^2$ remain with both speeds in large sets. Since the sets $V_j$ are disjoint and each large set contains at least $\eta n$ speeds, there are at most $1/\eta$ large sets. Therefore there are large sets $V_i,V_j$ such that 
\[
 e_{jk}\ge\varepsilon\eta^2n^2,
 \qquad \text{and}
 \qquad \rho_{jk}:=e_{jk}/(|V_j||V_k|) \ge\varepsilon\eta^2.
\]

We can now apply Proposition~\ref{mp:extraction} to $V_j,V_k$ and the threshold $D_{jk}$. The quantity $F_0$ in its statement satisfies
\[
 F_0\ge f':=\frac{(\varepsilon\eta^2)^q\eta^2}{4C^2}K^{-\pi(p_0)}>0,
\]
since $D_{jk}^2/(2^j2^k)=2^{|j-k|}/(4C^2n^2) \ge1/(4C^2n^2)$. We obtain a subset $U\subset V$ of size $m\geq\theta n$ with
\[
\theta=\frac\eta4(f'/K_0)^{24}>0,
\]
in a rational box of area at most $C'm$, where $C'=\max\{1, 2^9/f'\}$. Set $n_0=\lceil2/\theta\rceil$. If $n\ge n_0$, then $m\ge2$. Lemma~\ref{lem:dense rational box} and the transfer in Lemma~\ref{lem:Type I superadditive} give
\[
 R_V(1/2n)\ge\frac mn\kappa(C') \ge\theta\kappa(C')>0.
\]
For $2\le n<n_0$, the hypothesis ensures at least one pair of reduced ratio height at most $Cn$. Its common center intersections have total measure at least $1/(Cn^2)\ge1/(Cn_0^2)$, and this is bounded above by the Type I loss. Hence the theorem holds for every $n\ge2$ with
\[
 \kappa(C,\varepsilon)=\min\{\theta\kappa(C'),
                                      1/(Cn_0^2)\},
\]
as desired. 
\end{proof}

\section{Proof of the main theorem}\label{sec:main proof}

We now combine the two parts of the argument. If many pairs of speeds have bounded reduced ratio, the arithmetic result already gives a positive uncovered measure from Type~I losses. Otherwise the pair estimates are nearly uniform, and we can apply the dilation spillover inequality for Type II losses.

\begin{proof}[Proof of Theorem~\ref{thm:main}]
Choose
\[
  \kappa_0=\min\{\kappa(64,1/64),5/96\},
\]
where $\kappa(C,\varepsilon)$ is the constant in Theorem~\ref{mp:linear Type I}. Let 
\[
                    \delta_0=\frac1{2n},\qquad\text{and}\qquad
                    \delta_1=\frac3{4n}. 
\]
By Lemma~\ref{lem: base identities}, $g(V)\ge R_V(\delta_0)$. Thus, if $R_V(\delta_0)\ge\kappa(64,1/64)$, the desired measure bound follows immediately.

We may therefore assume that $R_V(\delta_0)<\kappa(64,1/64)$. Theorem~\ref{mp:linear Type I} implies that fewer than $n^2/64$ unordered pairs satisfy $H(u,v)\le64n$. Let $e$ be the number of such ordered pairs, so $e<n^2/32$. Every remaining pair contributes at least $(1-1/32)/n^2$ to $Q_V(\delta_0)$, by Lemma~\ref{lem: type ii}. The other contributions are nonnegative.
Hence
\begin{equation}\label{eq:baseline pair lower}
 Q_V(\delta_0)
 \ge\left(1-\frac1{32}\right)
       \left(1-\frac1n-\frac e{n^2}\right)
 \ge1-\frac1n-\frac1{16}.
\end{equation}

Next, let us compare the Type II overlaps. Write $w_j=w_V(\delta_j)$. We have $w_0=g(V)$. At radius $\delta_1$, the sum of the Bohr set measures is $3/2$, while their union has measure at most one. It follows that $w_1\ge1/2$. Since $n\ge2$, both radii are less than $1/2$, and hence Proposition~\ref{spill:pair spillover} applies and gives
\[
         Q_V(\delta_1) \ge\frac32Q_V(\delta_0)+1-3g(V).
\]
On the other hand, the upper bound in Lemma~\ref{lem: type ii} gives
\[
                 Q_V(\delta_1)\le\frac94\left(1-\frac1n\right).
\]
Combining these two estimates with \eqref{eq:baseline pair lower}, we obtain
\[
 3g(V)\ge1-\frac34\left(1-\frac1n\right)-\frac3{32}
                                                    \ge\frac5{32}.
\]
Thus $g(V)\ge5/96\ge\kappa_0$, proving the measure bound in the
remaining case.

It remains to deduce the improvement in maximum loneliness $\ML(V)$. Increase the radius from $\delta_0$ to
\[
 \delta=\frac{1+\kappa_0/2}{2n}.
\]
The sum of the added interval lengths is $2n(\delta-\delta_0)=\kappa_0/2$. The union can increase by no more than this amount, so its complement still has measure at least $\kappa_0/2>0$. Any element $t$ in this complement satisfies $\|vt\|>\delta$ for every $v\in V$. Taking $c=\kappa_0/2$ proves the theorem.
\end{proof}

\bibliographystyle{amsalpha}
\bibliography{reference}
\end{document}